\documentclass[12pt]{amsart}
\usepackage[normalem]{ulem}
\usepackage{url}

\usepackage{verbatim}
\usepackage{amssymb}
\usepackage{upref}
\usepackage[all]{xy}
\usepackage[usenames,dvipsnames]{color}
\usepackage{url}
\usepackage[normalem]{ulem}

\allowdisplaybreaks

\newcommand{\ignore}[1]{\relax}

\newtheorem{thm}{Theorem}
\newtheorem*{thm*}{Theorem}
\newtheorem{lem}[thm]{Lemma}

\theoremstyle{definition}

\newtheorem*{notn*}{Notation}

\newtheorem*{hyp*}{Hypothesis}

\newtheorem*{rem*}{Remark}

\numberwithin{equation}{section}

\newcommand{\midtext}[1]{\quad\text{#1}\quad}

\renewcommand{\and}{\midtext{and}}

\renewcommand{\epsilon}{\varepsilon}

\DeclareMathOperator*{\spn}{span}
\DeclareMathOperator*{\clspn}{\overline{\spn}}

\newcommand{\id}{\text{\textup{id}}}

\renewcommand{\iff}{\ensuremath{\Leftrightarrow}}
\newcommand{\then}{\ensuremath{\Rightarrow}}

\newcommand{\cs}%
{\ensuremath{\mathbf{C^*}}}
\newcommand{\csnd}%
{\ensuremath{\cs\!\!_\mathbf{nd}}}
\newcommand{\coact}%
{\ensuremath{\mathbf{C^*coact}}}
\newcommand{\coactnd}%
{\ensuremath{\mathbf{Coact}_\mathbf{nd}}}
\newcommand{\coactn}%
{\ensuremath{\coact^\mathbf{n}}}
\newcommand{\coactnnd}%
{\ensuremath{\coactn_\mathbf{nd}}}
\newcommand{\coactm}%
{\ensuremath{\coact^\mathbf{m}}}
\newcommand{\coactmnd}%
{\ensuremath{\coactm_\mathbf{nd}}}

\begin{document}
\title[Erratum to ``Full and reduced $C^*$-coactions'']{Erratum to ``Full and reduced $C^*$-coactions''.  Math.\ Proc.\ Camb.\ Phil.\ Soc.\ \textbf{116} (1994), 435--450.}

\author[Kaliszewski]{S. Kaliszewski}
\address{School of Mathematical and Statistical Sciences
\\Arizona State University
\\Tempe, Arizona 85287}
\email{kaliszewski@asu.edu}

\author[Quigg]{John Quigg}
\address{School of Mathematical and Statistical Sciences
\\Arizona State University
\\Tempe, Arizona 85287}
\email{quigg@asu.edu}



\date{\today}

\maketitle

Proposition~2.5 of \cite{fullred} states that a full coaction of a locally compact group
on a $C^*$-algebra is nondegenerate if and only if its normalization is.
Unfortunately,
the proof there only addresses the forward
implication,
and we have not been able to find a proof of the opposite implication.
This issue is important because 
the theory of crossed-product duality for coactions
requires implicitly that the coactions involved
be nondegenerate.
Moreover, each type of coaction
--- full, reduced, normal, maximal, and (most recently) exotic ---
has its own distinctive properties with respect
to duality, making it crucial
to be able to convert from one to the other without
losing nondegeneracy.

While it is generally believed that
all coactions of all types are nondegenerate, 
in this note we summarize what little is actually known about 
nondegeneracy of $C^*$-coactions.  We also hope to 
caution the reader that the error in 
\cite{fullred} has propagated
widely, and sometimes invisibly, in the literature. 
For example,
a \emph{normal} coaction is nondegenerate if and only if
the associated \emph{reduced} coaction is nondegenerate
(\cite[Proposition~3.3]{fullred}, which is independent
of Proposition~2.5).  
So since {reduced} coactions of \emph{discrete}
groups are automatically nondegenerate (\cite[Corollaire~7.15]{bs:hopf}),
it is often mistakenly assumed 
(as in \cite{Qdiscrete} and \cite{enchilada})
that every \emph{full} coaction of a {discrete} group is also nondegenerate.
An equivalent assumption 
(as in \cite[Section~2.4]{ks})
is that every $C^*$-algebra that carries a coaction of a discrete group
is the closed span of its spectral subspaces.

An overview of the definitions and basic results concerning 
full coactions,  their normalizations, and 
their reductions can be found in Appendix~A of \cite{enchilada}.

\begin{thm}\label{correct}
Let $(A,\delta)$ be a full coaction  of a locally compact group~$G$.
Then among the following conditions, we have the implications
\textup{(1)}~$\then$~\textup{(2)}~\textup{$\iff$~(3)}\textup{:}

\begin{enumerate}
\item $(A,\delta)$ is nondegenerate.
\item The normalization $(A^n,\delta^n)$ is nondegenerate.
\item The reduction $(A^r,\delta^r)$ is nondegenerate.
\end{enumerate}
\end{thm}

\begin{lem}\label{quotient}
Let $(A,\delta)$ and $(B,\epsilon)$ be full coactions of a locally compact group $G$.
If $(A,\delta)$ is nondegenerate 
and there exists a $\delta-\epsilon$ equivariant surjection $\varphi:A\to B$,
then $(B,\epsilon)$ is also nondegenerate.
\end{lem}

\begin{proof}
By equivariance,
$\epsilon(B) = \epsilon(\varphi(A)) = (\varphi\otimes\id)(\delta(A))$, so
\begin{align*}
\clspn \epsilon(B)(1\otimes C^*(G))
&=\clspn (\varphi\otimes\id)(\delta(A))(1\otimes C^*(G))\\
&=\clspn (\varphi\otimes\id)\bigl(\delta(A)(1\otimes C^*(G))\bigr)\\
&=(\varphi\otimes\id)\bigl(\clspn \delta(A)(1\otimes C^*(G))\bigr).
\end{align*}
Since $(A,\delta)$ is nondegenerate, 
$\clspn \delta(A)(1\otimes C^*(G) = A\otimes C^*(G)$, so
\[
\clspn \epsilon(B)(1\otimes C^*(G))
=(\varphi\otimes\id)(A\otimes C^*(G))
=B\otimes C^*(G).
\]
Thus $(B,\epsilon)$ is nondegenerate as well.
\end{proof}

\begin{proof}[Proof of Theorem~\ref{correct}]
Since $A^n$ is by definition
a quotient of $A$, and $\delta^n$ is defined
so that the quotient map is $\delta-\delta^n$ equivariant,
(1)~$\then$~(2) is immediate from Lemma~\ref{quotient}.  

By Definition~3.5 of~\cite{fullred}, 
the reduction $(A^r,\delta^r)$ 
of an arbitrary coaction $(A,\delta)$ coincides 
with the reduction $((A^n)^r,(\delta^n)^{r})$ 
of the normalization $(A^n,\delta^n)$.
By Proposition~3.3 of~\cite{fullred},
$(A^n,\delta^n)$ is nondegenerate
if and only if $((A^n)^r,(\delta^n)^{r})$ is,
and (2)~$\iff$~(3) follows.  
\end{proof}

We reiterate that it
is still an open question
whether or not (2)~implies (1).
We are grateful to Iain Raeburn for pointing out the error
in the original proof,
and to Alcides Buss for drawing our attention to 
a problem with our initial attempt to correct~it.


\providecommand{\bysame}{\leavevmode\hbox to3em{\hrulefill}\thinspace}
\providecommand{\MR}{\relax\ifhmode\unskip\space\fi MR }
\providecommand{\MRhref}[2]{%
  \href{http://www.ams.org/mathscinet-getitem?mr=#1}{#2}
}
\providecommand{\href}[2]{#2}

\end{document}